\documentclass[11pt]{amsart}

\usepackage{stmaryrd}
\usepackage{amscd,amsxtra,amssymb,mathrsfs, bbm}

\usepackage{multirow}
\usepackage[pdftex]{hyperref}
\usepackage{longtable}
\usepackage{tikz}
\usepackage{float}
\usepackage{color}
\usetikzlibrary{calc, positioning, shapes, fit, matrix, decorations}
\usetikzlibrary{decorations.shapes}

\usepackage[cmtip]{xypic}
\usepackage[all]{xy}
\xyoption{arc}

\newtheorem{theorem}{Theorem}[section]
\newtheorem{corollary}[theorem]{Corollary}
\newtheorem{lemma}[theorem]{Lemma}
\newtheorem{proposition}[theorem]{Proposition}

\theoremstyle{definition}
\newtheorem{definition}[theorem]{Definition}
\newtheorem{remark}[theorem]{Remark}
\newtheorem{example}[theorem]{Example}

\numberwithin{equation}{section}

\DeclareMathAlphabet{\mathpzc}{OT1}{pzc}{m}{it}

\DeclareMathOperator{\SL}{\mathsf{SL}}
\DeclareMathOperator{\SO}{\mathsf{SO}}

\renewcommand{\dim}{\mathsf{dim}}

\DeclareMathOperator{\Rep}{\mathsf{Rep}}
\DeclareMathOperator{\RepQ}{\mathsf{RepQ}}

\DeclareMathOperator{\HC}{\mathsf{HC}}

\DeclareMathOperator{\Hom}{\mathsf{Hom}}

\DeclareMathOperator{\GL}{\mathsf{GL}}
\DeclareMathOperator{\Aut}{\mathsf{Aut}}

\DeclareMathOperator{\End}{\mathsf{End}}

\DeclareMathOperator{\Spec}{\mathsf{Spec}}

\input xy
\xyoption{all}

\newcommand{\kk}{\mathbbm{k}}

\newcommand{\llbrace}{(\!(}
\newcommand{\rrbrace}{)\!)}

\newcommand{\II}{\mathbb{I}}

\renewcommand{\mod}{\mathsf{mod}}

\newcommand{\lar}{\longrightarrow}
\newcommand{\lieg}{\mathfrak{g}}

\newcommand{\RR}{\mathbb R}
\newcommand{\CC}{\mathbb C}

\newcommand{\cC}{\mathsf{C}}
\renewcommand{\cD}{\mathsf{D}}

\newcommand{\cE}{\mathsf{E}}

\newcommand{\Ad}{\mathsf{Ad}}

\newcommand{\idm}{\mathfrak{m}}

\newcommand{\EE}{\mathbb{E}}

\newcommand{\ZZ}{\mathbb{Z}}

\newcommand{\DD}{\mathbb{D}}

				\def\ga{\gamma}

\def\8{\infty}			
	\def\+{\oplus}		
\def\*{\otimes}

\def\ga{\gamma}

	\def\NN{\mathbb N}

\def\DMO{\DeclareMathOperator}
\DMO{\ob}{Ob}            \DMO{\mor}{Mor}
\DMO{\Ker}{Ker}
\DMO{\id}{Id}

\newcommand{\rightarrowdbl}{\longrightarrow\mathrel{\mkern-14mu}\rightarrow}

\title[Representation theory of the real  Gelfand order]{Representation theory of the real Gelfand order and real Harish-Chandra modules for $\SL_2(\RR)$}

\author{Igor Burban}
\address{
Universit\"at Paderborn,
Institut f\"ur Mathematik,
Warburger Strasse 100,
33098 Paderborn,
Germany
}
\email{burban@math.uni-paderborn.de} 

\author{Yuriy Drozd}
\address{
Institute of Mathematics, National Academy of
Sciences of Ukraine, Tereschenkivska str. 3, 01024 Kyiv, Ukraine}
\email{y.a.drozd@gmail.com}


\begin{document}

\begin{abstract} In this article we study the principal block of the category of real Harish-Chandra modules for the group $\mathsf{SL}_2(\RR)$ and relate it to the category of finite dimensional modules over the so-called real Gelfand order. We describe several distinguished classes of the corresponding indecomposable representations.
\end{abstract}
\maketitle

\section{Introduction}

In his ICM talk in 1970, I.~Gelfand \cite{Gelfand} raised the question of giving an explicit description of the indecomposable finite dimensional nilpotent representations of the quiver
\begin{equation}\label{E:GelfandQuiver}
\xymatrix
{
- \ar@/^/[rr]^{a_{-}}  & &  \star \ar@/^/[ll]^{b_{-}}
 \ar@/_/[rr]_{b_{+}}
 & &
\ar@/_/[ll]_{a_{+}} +}  \quad a_{-} b_{-} = a_{+} b_{+}.
\end{equation}
The motivation for studying this problem comes from the fact that over the field  of complex numbers,  the category of such representations of
(\ref{E:GelfandQuiver}) is equivalent to the principal block of the category of \emph{complex} Harish-Chandra modules corresponding to the group
$\SL_2(\RR)$.

Let $\kk$ be a field, $R = \kk\llbracket t\rrbracket$ be the algebra of formal power series with coefficients in $\kk$  and $\idm = (t)$ be its maximal ideal. 
The category of finite dimensional nilpotent representations of the quiver (\ref{E:GelfandQuiver}) is equivalent to the 
category $\Rep(O)$ of finite dimensional modules over the so-called \emph{Gelfand order}
\begin{equation}\label{E:GelfandOrder}
O =
\left[\begin{array}{ccc}
R & \idm & R\\
\idm & R & R \\
\idm & \idm & R
\end{array}
\right]
\end{equation}
Nazarova and Roiter showed in \cite{NazRoi} that the category $\Rep(O)$ is of tame representation type. A complete classification of the corresponding indecomposable objects was obtained independently by Bondarenko \cite{mp1, mp} and Crawley-Boevey \cite{CB}. In \cite{Nodal} the authors showed that the derived category $D^b\bigl(\Rep(O)\bigr)$ is also of tame representation type. Using this approach Burban and Gnedin recently obtained an explicit description of the indecomposable objects of $\Rep(O)$ in \cite{BurbanGnedin}.

From now on, let $\kk = \RR$.  The \emph{real Gelfand order} $A$ is defined as follows:
\begin{equation}\label{E:RealGelfand}
H = \left[\begin{array}{ccc}
R  &  R  & \idm\\
R  &  R   &  \idm \\
R  &  R & R
\end{array}
\right] \supset A = 
\left\{\left. 
\left(\begin{array}{ccc} a_{11} & a_{12} & a_{13} \\ a_{21} & a_{22} & a_{23} \\ a_{31} & a_{32} & a_{33}
\end{array}\right) \in  H \; \right|\;  \bar{a}_{11}= \bar{a}_{22},  \bar{a}_{21} = - \bar{a}_{12}\right\},
\end{equation}
where we set  $\bar{p} = p(0) \in \RR$ for any  $p \in R$.
It can be shown that $\CC \otimes_{\RR} A$ is isomorphic to the Gelfand order $O$. Moreover, $\Rep(A)$ is equivalent to the principal block of the category of \emph{real} Harish-Chandra modules associated with the group $\SL_2(\RR)$; see Theorem \ref{T:BlockDescr}.
The main results of this work are the following.  We explicitly describe 
\begin{enumerate}
\item[(a)] All Schurian $A$-modules; see Theorem \ref{T:Schurian}.
\item[(b)] All absolutely cyclic $A$-modules; see Theorem \ref{T:AbsCyclic}.
\end{enumerate} 
We also explain several well-known constructions concerning Galois descent in the categorical setting in Section \ref{SC:CatConstr};  this elaboration  appears to be of independent interest.

\smallskip
\noindent
\emph{List of notation}. For convenience of the reader we collect the most important
notation used in this paper.
\begin{itemize}
\item $\NN = \{1, 2, 3, \dots\}$ and $\NN_0 = \NN \cup \{0\}$.
\item $R = \RR\llbracket t \rrbracket$ with maximal ideal $\idm = (t)$. The symbol $K$
denotes the field of real Laurent series $\RR\llbrace t \rrbrace$ in Sections 2--5,
whereas $K = \SO_2(\RR)$ (with complexification $\tilde{K} = \SO_2(\CC)$) in Section 6.
\item $O$ is the complex Gelfand order (\ref{E:GelfandOrder}) and $A \subset H$ is the real
Gelfand order inside its hereditary envelope (\ref{E:RealGelfand});
$J = \mathsf{rad}(A) = \mathsf{rad}(H)$.
\item $\Rep(A)$ is the category of finite dimensional left $A$-modules and $\RepQ(A)$ its
linear algebra model (Definition \ref{D:RealGelfandLA}). The dimension vector of
$M \in \mathsf{Ob}\bigl(\Rep(A)\bigr)$ is the pair $(d_1, d_2) \in \NN_0^2$ with
$[M] = d_1 [S] + d_2 [T]$ in $K_0\bigl(\Rep(A)\bigr)$, where $S$ and $T$ are the two
simple $A$-modules, $\End_A(S) = \CC$ and $\End_A(T) = \RR$.
\item $C = \left\{ \left(\begin{smallmatrix} \alpha & -\beta \\ \beta & \alpha
\end{smallmatrix}\right) \middle|\; \alpha, \beta \in \RR \right\}$ is
the field of complex numbers $\CC$ regularly embedded into the algebra of square real matrices $M_2(\RR)$.
\item $P = Ae$, $Q = Af$ and $L$ are the three indecomposable $A$-lattices
(Lemma \ref{L:RR}).
\item $\DD$ denotes the duality (\ref{E:TwistedDuality}). Note that $(-)^\ddagger$ is
used for the complex conjugation in two different incarnations: for complex linear maps in real presentation as in Definition
\ref{D:RealGelfandLA}, and for the complex conjugation functor (\ref{E:ComplexConj}).
\item An object $M$ of $\Rep(A)$ is \emph{Schurian} if $\End_A(M)$ is a division algebra
over $\RR$ (Section 3), and \emph{absolutely cyclic} if its top is simple (Section 4).
\item For a $\kk$-linear category $\cC$, a $\kk$-algebra $\Gamma$ and a finite group $G$:
$\cC_\Gamma$ is the hom-extension of $\cC$ along $\Gamma$,  $\Gamma[G, \phi]$ the crossed product
(Definition \ref{D:GroupAction}) and $\cC[G, \Phi]$ the category attached to a
$G$-action on $\cC$; see Section \ref{SC:CatConstr}.
\item $\cD^\omega$ is the idempotent
completion of an additive category $\cD$.
\item $\lieg = \mathfrak{sl}_2(\RR)$ and $\tilde\lieg = \mathfrak{sl}_2(\CC)$; $c$ is the
Casimir element, $\HC(\lieg, K)$ the category of Harish-Chandra modules and
$\HC_\circ(\lieg, K)$ its principal block (Section \ref{S:HCModules}).
\end{itemize}

\smallskip
\noindent
\emph{Acknowledgement}. This work  was partially  supported by the  German Research Foundation SFB-TRR 358/1 2023 -- 491392403. The first-named author is grateful to Bill Crawley-Boevey for communicating to him another proof of  Lemma \ref{L:MoritaEq} and to Fabian Januszewski for helpful discussion of the arithmetic theory of Harish-Chandra modules. We also thank the anonymous referee for their comments.

\section{Representation of the real Gelfand order as a problem of linear algebra}

In what follows $A$ is the real Gelfand order, that is the $\RR$-algebra defined by  (\ref{E:RealGelfand}). We note the following facts.
\begin{itemize}
\item[(a)] $H$ is a \emph{hereditary order};  see \cite{ReinerMO} for the definition and basic facts about orders and, in particular, about hereditary orders. 
\item[(b)] $
J = 
\left[\begin{array}{ccc}
\idm  &  \idm  & \idm\\
\idm  &  \idm   &  \idm \\
R  &  R & \idm
\end{array}
\right] = \mathsf{rad}(A) = \mathsf{rad}(H)
$ is the common Jacobson radical of the orders  $A$ and $H$.
\item[(c)] We have: $A/J \cong \CC \times \RR$ and $H/J \cong M_2(\RR) \times \RR$. The embedding 
$A/J \lar H/J$  can be identified with $\CC \times \RR \xrightarrow{\varrho \times \mathsf{id}} M_2(\RR) \times \RR$, where $\CC \stackrel{\varrho}\lar M_2(\RR)$ is given by $\alpha + i\beta \mapsto \left(\begin{array}{cc} \alpha & - \beta \\ \beta & \alpha \end{array}\right)$ for any $\alpha, \beta \in \RR$. 
\end{itemize}
These properties show that $A$ is a \emph{real} \emph{nodal order}; see \cite{NodalFirst, NodalOrders} for the definition and main properties of this class of $\RR$-algebras. In what follows  $\Rep(A)$ denotes  the category of finite dimensional $A$-modules.

\begin{definition}\label{D:RealGelfandLA} We consider the following category $\RepQ(A)$.

\smallskip
\noindent
(I) Its objects are diagrams \begin{equation}\label{E:LinData}
\xymatrix
{
V \ar@/^/[rr]^-{\left(\begin{smallmatrix} Y_1 \\Y_2\end{smallmatrix}\right)}  & &  U \oplus U \ar@/^/[ll]^-{\left(\begin{smallmatrix} X_1  & X_2\end{smallmatrix}\right)}
 }
\end{equation}
where $V$ and $U$ are finite dimensional vector spaces over $\RR$ and 
$V \stackrel{X_i}\lar U$, $U \stackrel{Y_j}\lar V$, $1 \le i, j \le 2$ are linear maps  such that 
\begin{enumerate}
\item[(a)] $X_1 Y_1 = X_2 Y_2$ is nilpotent.
\item[(b)] $X_1 Y_2 = 0 = X_2 Y_1$. 
\end{enumerate}
 We call the pair $\bigl(\dim_\RR(U),\dim_\RR(V)\bigr)$ the \emph{dimension vector} of this representation.

\smallskip
\noindent
(II) A morphism 
\begin{equation*}
\left[\xymatrix
{
V \ar@/^/[rr]^-{\left(\begin{smallmatrix} Y_1 \\Y_2\end{smallmatrix}\right)}  & &  U \oplus U \ar@/^/[ll]^-{\left(\begin{smallmatrix} X_1  & X_2\end{smallmatrix}\right)}
 }
 \right] \xrightarrow{(S, T)}
 \left[\xymatrix
{
V' \ar@/^/[rr]^-{\left(\begin{smallmatrix} Y'_1 \\Y'_2\end{smallmatrix}\right)}  & &  U' \oplus U' \ar@/^/[ll]^-{\left(\begin{smallmatrix} X'_1  & X'_2\end{smallmatrix}\right)}
 }
 \right] 
\end{equation*} 
is given by a pair of linear maps $V \stackrel{S}\lar V'$ and $U \oplus U \stackrel{T}\lar U' \oplus U'$ such that
$T = \left(\begin{array}{cc} T_1 & -T_2 \\ T_2 & T_1\end{array}\right)$ for some linear maps $U \stackrel{T_k}\lar U',k = 1, 2$ and the following diagrams of vector spaces and linear maps 
\begin{equation}
\begin{array}{ccc}
\xymatrix
{
V \ar[rr]^-{\left(\begin{smallmatrix} Y_1 \\Y_2\end{smallmatrix}\right)} \ar[d]_-{S} & &  U \oplus U \ar[d]^-{T}\\
V' \ar[rr]^-{\left(\begin{smallmatrix} Y'_1 \\Y'_2\end{smallmatrix}\right)}  & &  U' \oplus U'
 } & \begin{array}{c} \\ \\ \\ \mbox{\rm and} \\ \end{array} & 
 \xymatrix
{
V  \ar[d]_-{S} & &  U \oplus U \ar[d]^-{T^\ddagger} \ar[ll]_-{\left(\begin{smallmatrix} X_1 & X_2\end{smallmatrix}\right)}\\
V'   & &  U' \oplus U' \ar[ll]^-{\left(\begin{smallmatrix} X'_1 & X'_2\end{smallmatrix}\right)}
 } 
 \end{array}
\end{equation}
are commutative, where $T^\ddagger = \left(\begin{array}{cc} T_1 & T_2 \\ -T_2 & T_1\end{array}\right)$.

\smallskip
\noindent
(III) The composition of morphisms is given by the composition of the corresponding linear maps. 
\end{definition}

\begin{proposition}\label{GelfandOrderLAProblem} The categories $\Rep(A)$ and $\RepQ(A)$ are equivalent. 
\end{proposition}

\begin{proof} We consider the following elements of the $\RR$-algebra $A$:
\begin{equation}\label{E:KeyElements1}
x_1 = 
\left( 
\begin{array}{ccc}
0 & 0 & 0 \\
0 & 0 & 0 \\
1 & 0 & 0 \\
\end{array}
\right), \, 
x_2 = 
\left( 
\begin{array}{ccc}
0 & 0 & 0 \\
0 & 0 & 0 \\
0 & 1 & 0 \\
\end{array}
\right), \, 
y_1 = 
\left( 
\begin{array}{ccc}
0 & 0 & t \\
0 & 0 & 0 \\
0 & 0 & 0 \\
\end{array}
\right), \, 
y_2 = 
\left( 
\begin{array}{ccc}
0 & 0 & 0 \\
0 & 0 & t \\
0 & 0 & 0 \\
\end{array}
\right)
\end{equation}
as well as 
\begin{equation}\label{E:KeyElements}
e = 
\left( 
\begin{array}{ccc}
1 & 0 & 0 \\
0 & 1 & 0 \\
0 & 0 & 0 \\
\end{array}
\right), \, 
\jmath = 
\left( 
\begin{array}{ccc}
0 & -1 & 0 \\
1 & 0 & 0 \\
0 & 0 & 0 \\
\end{array}
\right) \; \mbox{\rm and} \;  
f = 
\left( 
\begin{array}{ccc}
0 & 0 & 0 \\
0 & 0 & 0 \\
0 & 0 & 1 \\
\end{array}
\right).
\end{equation}
Then $x_1 y_1 = x_2 y_2$ and $x_1 y_2 = 0 = x_2 y_1$. Since $\jmath^2 = -e$, the $\RR$-algebra generated by $e$ and $\jmath$ is isomorphic to $\CC$. 

Recall that a complex structure on a real vector space $W$ is an $\RR$-linear map $W \stackrel{J}\lar W$ such that $J^2 = -I_W$. It is well known that one can choose an isomorphism $W \cong U \oplus U$ such that $J$ will be given by $\left(\begin{array}{cc} 0 & -I_U \\ I_U & 0
\end{array}\right)$. In these terms a $\CC$-linear map $(W, J) \stackrel{T}\lar (W', J')$ is given by a matrix of the form $T = \left(\begin{array}{cc} T_1 & -T_2 \\ T_2 & T_1\end{array}\right)$ for some $\RR$-linear maps $U \stackrel{T_k}\lar U'$, where $k = 1, 2$.

Let $\xymatrix
{
V \ar@/^/[rr]^-{\left(\begin{smallmatrix} Y_1 \\Y_2\end{smallmatrix}\right)}  & &  U \oplus U \ar@/^/[ll]^-{\left(\begin{smallmatrix} X_1  & X_2\end{smallmatrix}\right)}
 }$ be an object of the category $\RepQ(A)$. Then the real vector space $U \oplus U \oplus V$ carries the structure of a left $A$-module, with the actions of $x_k, y_k$ ($k = 1, 2$), $e, \jmath$ and $f$ given by the following $\RR$-linear endomorphisms: 
 \begin{equation*}
x_1 = 
\left( 
\begin{array}{ccc}
0 & 0 & 0 \\
0 & 0 & 0 \\
X_1 & 0 & 0 \\
\end{array}
\right), \, 
x_2 = 
\left( 
\begin{array}{ccc}
0 & 0 & 0 \\
0 & 0 & 0 \\
0 & X_2 & 0 \\
\end{array}
\right), \, 
y_1 = 
\left( 
\begin{array}{ccc}
0 & 0 & Y_1 \\
0 & 0 & 0 \\
0 & 0 & 0 \\
\end{array}
\right), \, 
y_2 = 
\left( 
\begin{array}{ccc}
0 & 0 & 0 \\
0 & 0 & Y_2 \\
0 & 0 & 0 \\
\end{array}
\right)
\end{equation*}
as well as 
\begin{equation*}
e = 
\left( 
\begin{array}{ccc}
I & 0 & 0 \\
0 & I & 0 \\
0 & 0 & 0 \\
\end{array}
\right), \, 
\jmath = 
\left( 
\begin{array}{ccc}
0 & -I & 0 \\
I & 0 & 0 \\
0 & 0 & 0 \\
\end{array}
\right) \; \mbox{\rm and} \;  
f = 
\left( 
\begin{array}{ccc}
0 & 0 & 0 \\
0 & 0 & 0 \\
0 & 0 & I \\
\end{array}
\right).
\end{equation*}
A straightforward computation shows that this assignment extends to a functor establishing an equivalence of categories 
$\RepQ(A) \stackrel{\EE}\lar \Rep(A)$. 
\end{proof}

\begin{remark}
Since $A/J \cong \CC \times \RR$, the category $\Rep(A)$ has precisely two simple objects, denoted by $S$ and $T$, where $\End_A(S) = \CC$ and $\End_A(T) = \RR$. Let $K_0\bigl(\Rep(A)\bigr)$ denote the Grothendieck group of $\Rep(A)$. Then $K_0\bigl(\Rep(A)\bigr) = \bigl\langle [S], [T]\bigr\rangle_{\ZZ} \cong \ZZ^2$. For any object $M$ of $\Rep(A)$, there exist unique $d_1, d_2 \in \NN_0$ such that
$[M] = d_1 [S] + d_2[T] \in K_0\bigl(\Rep(A)\bigr)$, so we define $\underline{\dim}(M) = (d_1, d_2)$ to be the dimension vector of $M$. This explains the notion of the dimension vector of an object of $\RepQ(A)$ introduced in Definition \ref{D:RealGelfandLA}.
\end{remark}

For any $\RR$-linear map $V_1 \stackrel{X}\lar V_2$, we denote by $V_2^\ast \stackrel{X^\ast}\lar V_1^\ast$ the corresponding dual map. Next, note that the opposite algebra of $A$ can be described as follows: 
\begin{equation*}
 A^\circ = 
\left\{\left. 
\left(\begin{array}{ccc} a_{11} & a_{12} & a_{13} \\ a_{21} & a_{22} & a_{23} \\ a_{31} & a_{32} & a_{33}
\end{array}\right) \in  M_3(R) \; \right|\;  
\begin{array}{c}
\left(
\begin{array}{cc}
\bar{a}_{11} & \bar{a}_{12} \\  \bar{a}_{21} & \bar{a}_{22}
\end{array}
\right) \in C
 \\
\bar{a}_{31}= \bar{a}_{32} = 0
\end{array}
\right\}.
\end{equation*}
where $C = \left\{\left.\left(\begin{array}{cc}
\alpha& -\beta \\ \beta & \alpha
\end{array}\right) \right| \alpha, \beta \in \RR \right\}$.
Indeed, the map $A \lar A^\circ$ assigning to $X \in A$ the transposed matrix $X^t \in A^\circ$ gives the corresponding identification of $A^\circ$ and the opposite algebra of $A$.  Next, consider the matrix $Z = \left(\begin{array}{ccc} 1 & 0 & 0 \\ 0 & 1 & 0 \\ 0 & 0& t
\end{array} \right)$. Then the map
$
A \stackrel{\mathsf{Ad}_Z}\lar  A^\circ, \, X \mapsto Z X Z^{-1}
$
is an isomorphism of $\RR$-algebras, hence it induces an equivalence of categories $\Rep(A^\circ) \xrightarrow{\mathsf{Ad}^\sharp_Z} \Rep(A)$. We have the duality functor $\Rep(A) \xrightarrow{\widehat\DD} \Rep(A^\circ), M \mapsto M^\ast$. In this way we obtain a contravariant auto-equivalence $\Rep(A) \xrightarrow{{\DD}}  \Rep(A)$ given by the composition
\begin{equation}\label{E:TwistedDuality}
\Rep(A) \xrightarrow{\widehat{\DD}} \Rep(A^\circ) \xrightarrow{\mathsf{Ad}^\sharp_Z} \Rep(A).
\end{equation}

The following result can be verified by a straightforward computation. 

\begin{lemma}
Under the equivalence of categories $\RepQ(A) \stackrel{\EE}\lar \Rep(A)$ constructed in Proposition \ref{GelfandOrderLAProblem}, the twisted duality functor (\ref{E:TwistedDuality}) is identified with the contravariant auto-equivalence
$\RepQ(A) \stackrel{\DD}\lar \RepQ(A)$ given by
\begin{equation}\label{E:Duality}
\left[\xymatrix
{
V \ar@/^/[rr]^-{\left(\begin{smallmatrix} Y_1 \\ Y_2\end{smallmatrix}\right)} & & U \oplus U \ar@/^/[ll]^-{\left(\begin{smallmatrix} X_1 & X_2\end{smallmatrix}\right)}
}\right]
\stackrel{\DD}\longmapsto
\left[\xymatrix
{
V^\ast \ar@/^/[rr]^-{\left(\begin{smallmatrix} X_1^\ast \\ X_2^\ast\end{smallmatrix}\right)} & & U^\ast \oplus U^\ast \ar@/^/[ll]^-{\left(\begin{smallmatrix} Y_1^\ast & Y_2^\ast\end{smallmatrix}\right)}
}\right].
\end{equation}
On morphisms, this functor is given by
$\DD(S, T) = (S^\ast, {T^\ddagger}^\ast)$,
where the $R$-linear map $T^\ddagger$ is as in Definition \ref{D:RealGelfandLA}.
\end{lemma}

\section{Schurian representations of the real Gelfand order}

Recall that $M \in \mathsf{Ob}\bigl(\Rep(A)\bigr)$ is called \emph{Schurian} if its endomorphism algebra $\End_A(M)$ is a division algebra over $\RR$.

\begin{theorem}\label{T:Schurian} The category $\Rep(A)$ contains precisely six Schurian objects:
\begin{enumerate}
\item[(a)] Two simple modules $S$ and $T$ of dimension vectors, respectively, $(1,0)$ and $(0,1)$.
\item[(b)] Two modules of length two:
\begin{equation}
\begin{array}{ccc}
\xymatrix
{
\RR \ar@/^/[rr]^-{\left(\begin{smallmatrix} 1 \\\hline \\ 0\end{smallmatrix}\right)}  & &  \RR \oplus\RR \ar@/^/[ll]^-{\left(\begin{smallmatrix} 0  \;  |& 0 \end{smallmatrix}\right)}
 } & 
 \mbox{\rm and} & \xymatrix
{
\RR \ar@/^/[rr]^-{\left(\begin{smallmatrix} 0 \\\hline \\ 0\end{smallmatrix}\right)}  & &  \RR \oplus\RR \ar@/^/[ll]^-{\left(\begin{smallmatrix} 1  \;  |& 0 \end{smallmatrix}\right)}
 } 
 \end{array}
\end{equation}
whose endomorphism algebras are isomorphic to $\RR$. 
\item[(c)] Two modules of length three:
\begin{equation}
\begin{array}{ccc}
\xymatrix
{
\RR^2 \ar@/^/[rr]^-{\tiny{\left(
\begin{array}{cc}
1 & 0\\
\hline
0 & 1
\end{array}
\right)
}}  & &  \RR \oplus\RR \ar@/^/[ll]^-{
\tiny{\left(
\begin{array}{c|c}
0 & 0\\
0 & 0
\end{array}
\right)
}}} & 
 \mbox{\rm and} & 
 \xymatrix
{
\RR^2 \ar@/^/[rr]^-{\tiny{\left(
\begin{array}{cc}
0 & 0\\
\hline
0 & 0
\end{array}
\right)
}}  & &  \RR \oplus\RR \ar@/^/[ll]^-{
\tiny{\left(
\begin{array}{c|c}
1 & 0\\
0 & 1
\end{array}
\right)
}}}
\end{array}
 \end{equation}
whose endomorphism algebras are isomorphic to $\CC$. 
\end{enumerate}
\end{theorem}

\begin{proof} It is clear that the six listed objects of $\Rep(A)$ are pairwise non-isomorphic and Schurian. It remains to show that this list is exhaustive.

Let $M$ be a finite dimensional $A$-module. Since $t \in A$ is central, the map $M \xrightarrow{\lambda_t} M, x \mapsto tx$ is $A$-linear and nilpotent. If $M$ is Schurian, then $\lambda_t = 0$, and hence $M$ is a module over the finite dimensional algebra $A/tA$. Consider the following $\RR$-algebra
\begin{equation}
\Lambda = 
\left\{\left. 
\left(\begin{array}{ccc} \bar{\alpha} & 0  & 0  \\ \beta & \rho & 0  \\ \delta & \gamma & \alpha
\end{array}\right) \in  M_3(\CC) \; \right|\;  
\alpha, \beta, \gamma, \delta \in \CC, \rho \in \RR
\right\}.
\end{equation}
One can check that the map
$$
\Lambda \lar A/tA, \left(\begin{array}{ccc} \bar{\alpha} & 0  & 0  \\ \beta & \rho & 0  \\ \delta & \gamma & \alpha
\end{array}\right) \mapsto
\left[\left(
\begin{array}{c| c}
\left(
\begin{array}{cc}
a_1 & -a_2 \\
a_2 & a_1
\end{array}
\right) + 
t \left(
\begin{array}{cc}
d_1 &  0 \\
d_2 & 0
\end{array} 
\right)
& \begin{array}{c}
t c_1 \\
t c_2
\end{array} 
\\
\hline
b_1 \quad \quad b_2 \quad \quad \quad  \quad \quad \quad  & \rho
\end{array}
\right)\right],
$$
where $\alpha = a_1 + i a_2$, $\beta = b_1 + i b_2$, $\gamma = c_1 + i c_2$ and $\delta =  d_1 + i d_2$, is an isomorphism of $\RR$-algebras. 

Next, consider $e = 
\left( 
\begin{array}{ccc}
1 & 0 & 0 \\
0 & 0 & 0 \\
0 & 0 & 1 \\
\end{array}
\right) \in \Lambda$ as well as the corresponding indecomposable projective $\Lambda$-module $P = \Lambda e$. Note that 
$\End_{\Lambda}(P) = e \Lambda e$ has non-zero radical; hence $P$ is not Schurian. 
We claim that $P$ is an injective $\Lambda$-module. Indeed, $P$ is of length 4 and
$$
I = \mathsf{soc}(P) = \left\langle 
\left( 
\begin{array}{ccc}
0 & 0 & 0 \\
0 & 0 & 0 \\
1 & 0 & 0 \\
\end{array}
\right)
\right\rangle_{\CC}
$$
is a simple $\Lambda$-module. Consequently, the injective envelope $E_P$ of $P$ is indecomposable. However, the indecomposable injective $\Lambda$-modules have lengths 2 and 4. Therefore, the monomorphism $P \lar E_P$ must be an isomorphism, hence $P$ is injective, as asserted. 

Note that $I$ is a two-sided ideal in $\Lambda$. Let $\Gamma = \Lambda/I$. By the Lemma on Separation, see \cite[Section 9.2]{DrozdKirichenko}, we have: 
$$
\mathsf{Ind}(\Lambda\mbox{\rm --}\mathsf{mod}) = \{P\} \sqcup \mathsf{Ind}(\Gamma\mbox{\rm --}\mathsf{mod}).
$$
Hence, all Schurian objects of $\Rep(A)$ arise from Schurian objects of $\Gamma\mbox{--}\mod$. Let $J = \mathsf{rad}(\Gamma)$ be the Jacobson radical of $\Gamma$. It is clear that $J^2 = 0$. We have:
$$
\Omega := 
\left[
\begin{array}{cc}
\Gamma/J & 0 \\
J & \Gamma/J
\end{array}
\right] \cong 
\left[
\begin{array}{cc}
\CC & 0 \\
_\RR\CC_{\bar{\CC}} & \RR
\end{array}
\right] \times \left[
\begin{array}{cc}
\RR & 0 \\
_\CC\CC_{\RR} & \CC
\end{array}
\right] =: \Omega_1 \times \Omega_2.
$$
Note that $\Omega$ is a hereditary $\RR$-algebra. By a result of Auslander and Reiten, see \cite[Theorem V.2.1]{AuslanderReiten}, we have an equivalence of stable categories:
$
\underline{\Gamma\mbox{\rm --}\mathsf{mod}} \simeq \underline{\Omega\mbox{\rm --}\mathsf{mod}}.
$ For both $k = 1, 2$, the $\RR$-algebra $\Omega_k$ is a hereditary $\RR$-algebra of type $B_2$. It follows that the  category $\Omega_k\mbox{\rm --}\mathsf{mod}$ has precisely four indecomposable objects, two of which are projective, see  \cite{Gabriel}. Therefore, $\Gamma\mbox{\rm --}\mathsf{mod}$ has precisely six indecomposable objects. As a consequence, the category  $\Rep(A)$ has at most six Schurian objects, so all such objects are precisely those listed in the statement of the theorem. 
\end{proof}

\section{Absolutely cyclic representations of the real Gelfand order}

\begin{definition} An object $M \in \mathsf{Ob}\bigl(\Rep(A)\bigr)$ is called \emph{absolutely cyclic}
if one of the following equivalent conditions holds:
\begin{enumerate}
\item[(a)] There exists an epimorphism $F \rightarrowdbl M$, where $F$ is an \emph{indecomposable} projective $A$-module.
\item[(b)] The top of $M$ is simple.
\item[(c)] $M$ has a unique maximal proper submodule. 
\end{enumerate}
Clearly, such a module $M$ is automatically indecomposable. 
\end{definition}
\begin{example}
The only Schurian object of $\Rep(A)$ which is \emph{not} absolutely cyclic is 
$\xymatrix{
\RR^2 \ar@/^/[rr]^-{\tiny{\left(
\begin{array}{cc}
1 & 0\\
\hline
0 & 1
\end{array}
\right)
}}  & &  \RR \oplus\RR \ar@/^/[ll]^-{
\tiny{\left(
\begin{array}{c|c}
0 & 0\\
0 & 0
\end{array}
\right)
}}}$.
 Its top is isomorphic to $T^{\oplus 2}$. 
\end{example}
\begin{remark} Absolutely cyclic representations of the complex Gelfand order (\ref{E:GelfandOrder}) were classified in \cite[Theorem 2.5]{ABR}. They correspond  to those complex Harish-Chandra modules for $\SL_{2}(\RR)$ which are associated with polyharmonic vector-valued Maa\ss{} forms for finite index subgroups of $\SL_2(\ZZ)$, with exponential growth allowed at the cusps.
\end{remark}

Recall that a finitely generated $A$-module $F$ is called a \emph{lattice} if it is torsion-free as an $R$-module. In what follows, we shall need the following result, whose proof follows from \cite[Theorem B]{RingelRoggenkamp}.

\begin{lemma}\label{L:RR}
There are precisely three non-isomorphic indecomposable $A$-lattices:
\begin{equation*}
P = A e = 
\left\{\left.\left(\begin{array}{cc} \alpha & - \beta \\ \beta & \alpha \\ 0 & 0\end{array}\right) \right| \alpha, \beta \in \RR\right\} + \left[\begin{array}{cc} \idm & \idm \\ \idm & \idm \\ R & R\end{array} \right], \; Q = A f = \left[\begin{array}{c} \idm \\ \idm \\ R \end{array} \right] \, \mbox{\rm{and}} \; L = \left[\begin{array}{c} R \\ R \\ R \end{array} \right].
\end{equation*}
\end{lemma}

\smallskip
\noindent
In particular, $P$ and $Q$ are just the indecomposable projective $A$-modules with $P/JP\simeq S$ and $Q/JQ\simeq T$.

\smallskip
\noindent
Let $K = \RR\llbrace t\rrbrace$. Then $\widetilde{A}:= K \otimes_R A \cong M_3(K)$ and $E = \left[\begin{array}{c} K \\ K \\ K\end{array} \right]$ is the unique indecomposable $\widetilde{A}$-module. Hence for any $A$-lattice $F$ there exists a unique $n \in \NN$ (called the \emph{rank} of $F$) such that the rational hull $K \otimes_R F$ of $F$ is isomorphic to $$E^{\oplus n} \cong \mathsf{Mat}_{3 \times n}(K) = 
\left[\begin{array}{ccc} K & \dots &  K \\  K & \dots &  K \\ K & \dots &  K \\  \end{array} \right]
$$ 
with the natural action of $\widetilde{A}$ given by left matrix multiplication.
Of course, the ranks of $Q$ and $L$ are one, whereas the rank of $P$ is two. Note that we have an isomorphism of $K$-vector spaces
$$
\mathsf{Mat}_{n \times m}(K) \stackrel{\rho}\lar \Hom_{\widetilde{A}}\bigl(E^{\oplus n}, E^{\oplus m}\bigr),
$$
where for any $X \in \mathsf{Mat}_{n \times m}(K)$ the corresponding $\widetilde{A}$-linear map
$E^{\oplus n} \xrightarrow{\rho_X} E^{\oplus m}$ is given by right multiplication with $X$. 
It follows that for any $A$-lattices $F'$ and $F''$ we have an embedding of $R$-modules
\begin{equation}\label{E:DescrMorphisms}
\Hom_A(F', F'') \lar \Hom_{\widetilde{A}}(K \otimes_R F', K \otimes_R F'') \cong \mathsf{Mat}_{n \times m}(K),
\end{equation}
where $n = \mathsf{rk}(F')$ and $m = \mathsf{rk}(F'')$. Next, an $A$-linear map $F' \stackrel{\phi}\lar F''$ is called a \emph{rational isomorphism} if the induced map of rational hulls $K \otimes_R F' \xrightarrow{\mathit{1} \otimes \phi} K \otimes_R F''$ is an isomorphism. Of course, any such $\phi$ is injective and its cokernel is finite dimensional over $\RR$. We denote by $\Hom_A^\circ(F', F'')$ the subset of $\Hom_A(F', F'')$ consisting of rational isomorphisms. 

\smallskip
\noindent
The proof of the following result is straightforward. 

\begin{lemma} In terms of the embedding (\ref{E:DescrMorphisms}), we have the following isomorphisms of $R$-algebras:
\begin{enumerate}
\item[(a)] $\End_A(L) \cong R \cong \End_A(Q)$.
\item[(b)] $\End_A(L \oplus Q) \cong \left[
\begin{array}{cc}
R & \idm \\
R & R
\end{array}
\right]$.
\item[(c)] $\End_A(P) = C  + \left[\begin{array}{cc} \idm & \idm \\ \idm & \idm \end{array} \right]$.
\end{enumerate}
\end{lemma}

\smallskip
\noindent
The following theorem is one of the main results of this paper. 
\begin{theorem}\label{T:AbsCyclic} 
Let $M$ be an absolutely cyclic object of $\Rep(A)$.
\begin{enumerate}
\item[(I)] If the top of $M$ is isomorphic to $T$, then $M$ is isomorphic to the cokernel of one of the following maps, for a uniquely determined $k \in \NN$:
\begin{enumerate}
\item[(a)] $Q \xrightarrow{t^k} Q$, and its dimension vector is $(k, k)$.
\item[(b)] $L \xrightarrow{t^k} Q$, and its dimension vector is $(k-1, k)$.
\end{enumerate}
\item[(II)] If the top of $M$ is isomorphic to $S$, then $M$ is isomorphic to the cokernel of one of the following maps: 
\begin{enumerate}
\item[(a)] $Q^{\oplus 2}  \xrightarrow{t^k \left(\begin{smallmatrix} 1 & 0 \\ 0 & t^l \end{smallmatrix}\right)} P$ for uniquely determined $k, l \in \NN_0$. Its dimension vector is $(2k+l+1,  2k +l)$.
\item[(b)] $L^{\oplus 2}  \xrightarrow{t^k \left(\begin{smallmatrix} 1 & 0 \\ 0 & t^l \end{smallmatrix}\right)} P$ for uniquely determined $k \in \NN$ and $l \in \NN_0$, and its dimension vector is $(2k+l-1, 2k +l)$.
\item[(c)] $L \oplus Q  \xrightarrow{t^k \left(\begin{smallmatrix} 1 & 0 \\ 0 & t^l\end{smallmatrix}\right)} P$ for some $k \in \NN$ and $l \in \NN_0$, or $L \oplus Q  \xrightarrow{t^k \left(\begin{smallmatrix} t^l & 0 \\ 0 &1 \end{smallmatrix}\right)} P$ for some  $k \in \NN_0$ and $l \in \NN$. In both cases, such $k$ and $l$ are uniquely determined, and the dimension vector is $(2k+l, 2k+l)$. 
\item[(d)] $P  \xrightarrow{t^k \left(\begin{smallmatrix} 1 & 0 \\ 0 & \lambda t^l\end{smallmatrix}\right)} P$ for some $k \in \NN$ and $l \in \NN_0$ and $\lambda \in \RR^\ast$. The parameters $k$ and $l$ are uniquely determined, and the dimension vector is $(2k+l, 2k+l)$. 
\begin{enumerate}
\item[(i)] If $l = 0$, then the modules corresponding to $\lambda$ and $\lambda'$ are isomorphic if and only if $\lambda' \in \bigl\{\lambda, \lambda^{-1}\bigr\}$. Hence, we take $\lambda \in [-1, 1] \setminus \{0\}$ to ensure uniqueness.
\item[(ii)] If $l \ge 1$, then the modules corresponding to different $\lambda$ are non-isomorphic. 
\end{enumerate}
\end{enumerate}
\end{enumerate}
\end{theorem}

\begin{proof} (I) By assumption, we have an $A$-linear epimorphism $Q \stackrel{\pi}\rightarrowdbl M$. The kernel $F$ of $\pi$ is an $A$-lattice of rank one. By Lemma \ref{L:RR}, we have: $F \cong Q$ or $F \cong L$, which implies the first part of the statement. 

\smallskip
\noindent
(II) By assumption, we have an $A$-linear epimorphism $P \stackrel{\pi}\rightarrowdbl M$. The kernel $F$ of $\pi$ is an $A$-lattice of rank two, and we have a short exact sequence
\begin{equation}
0 \lar F \stackrel{\phi}\lar P \stackrel{\pi}\lar M \lar 0.
\end{equation}
 By Lemma \ref{L:RR}, we have:
$F \in \bigl\{Q \oplus Q, L \oplus L, L \oplus Q, P\bigr\}$. In each of these cases, it remains to describe the normal form of $\phi \in \Hom_A^\circ(F, P)$ under the natural action of the group $\Aut_A(F) \times \Aut_A(P)$, since such orbits correspond precisely to the isomorphism classes of absolutely cyclic modules with simple top $S$. In terms of the identification (\ref{E:DescrMorphisms}) of the morphism spaces, we obtain a ``matrix problem''
\begin{equation}\label{E:MatrixProblem}
\phi \mapsto \eta  \, \phi \,  \xi, \; \mbox{\rm where} \; \phi\in \Hom_A^\circ(F, P), \eta \in \Aut_A(F) \;  \mbox{\rm and} \; \xi \in \Aut_A(P). 
\end{equation}
A straightforward computation shows that 
$$
\Hom_A(L, P) \cong \bigl[\idm \; \idm \bigr] \; \mbox{\rm whereas } \; 
\Hom_A(Q, P) \cong \bigl[R \; R\bigr]. 
$$
It follows that
$
\Hom_A(Q^{\oplus 2}, P) = \left[\begin{array}{cc}
R & R \\ R & R
\end{array} \right]$,  $\Hom_A(L^{\oplus 2}, P) = \left[\begin{array}{cc}
\idm & \idm \\ \idm  & \idm
\end{array} \right] = t \left[\begin{array}{cc}
R & R \\ R & R
\end{array} \right]$, whereas   $
\Hom_A(L \oplus Q, P) = \left[\begin{array}{cc}
\idm & \idm \\ R & R
\end{array} \right].
$

\smallskip
\noindent
(a) Suppose that $F \cong Q^{\oplus 2}$. Then $\phi \in \Hom^\circ_A(Q^{\oplus 2}, P) \subset \left[\begin{array}{cc}
R & R \\ R & R
\end{array} \right]$ is a matrix whose determinant is non-zero. There exists a uniquely determined $k \in \NN_0$ such that 
\begin{equation}\label{E:FirstRed}
\phi = t^k 
\left( 
\begin{array}{cc}
a & b \\
c & d
\end{array}
\right) + t^{k+1} \left( 
\begin{array}{cc}
a' & b' \\
c' & d'
\end{array}
\right),
\end{equation}
where $0 \ne \left( 
\begin{array}{cc}
a & b \\
c & d
\end{array}
\right) \in M_2(\RR)$ and $\left( 
\begin{array}{cc}
a' & b' \\
c' & d'
\end{array}
\right) \in M_2(R)$. Since $\left( 
\begin{array}{cc}
0 & 1 \\
1 & 0
\end{array}
\right) \in \GL_2(R) = \Aut_A(Q^{\oplus 2})$, we may assume, without loss of generality, that $(a, b) \ne (0, 0)$. Next, there exists $\xi \in C$ such that 
$
\left( 
\begin{array}{cc}
a & b \\
c & d
\end{array}
\right) \xi = 
\left( 
\begin{array}{cc}
1 & 0 \\
\ast & \ast
\end{array}
\right).
$
The transformation rule (\ref{E:MatrixProblem}) allows one to add for any $p \in R$ any $p$-multiple of any 
row of $\phi$ to another row of $\phi$. Therefore, we have: $\phi \sim t^k \left(\begin{array}{cc} 1 & 0 \\ 0 & t^l \end{array}\right)$ for some $k, l \in \NN_0$. Since $\mathsf{rad}(P) \cong Q^{\oplus 2}$, the morphism $\phi$ admits, up to automorphisms of the source and the target, the following factorization: 
$$
Q^{\oplus 2} \xrightarrow{t^k \left(\begin{smallmatrix} 1 & 0 \\ 0 & t^l \end{smallmatrix}\right)} Q^{\oplus 2} \subset P.
$$
It follows that dimension vector of $\mathsf{Cok}(\phi)$ is $(2k+l+1)[S] + (2k +l)[T]$. 

\smallskip
\noindent
(b) Now assume that $F \cong L^{\oplus 2}$. As in case (a), one can show that for any $\phi \in \Hom^\circ_A(L^{\oplus 2}, P) \subset  t \left[\begin{array}{cc}
R & R \\ R & R
\end{array} \right]$ there exist $k \in \NN$ and $l \in \NN_0$ such that $\phi \sim t^k \left(\begin{array}{cc} 1 & 0 \\ 0 & t^l \end{array}\right)$. Again, up to automorphisms of the source and the target, $\phi$ admits a factorization
$$
L^{\oplus 2} \xrightarrow{t^k \left(\begin{smallmatrix} 1 & 0 \\ 0 & t^l \end{smallmatrix}\right)} Q^{\oplus 2} \subset P.
$$
It follows that the dimension vector of $\mathsf{Cok}(\phi)$ is $(2k+l-1, 2k +l)$. 

\smallskip
\noindent
(c) We now consider the case $F \cong L \oplus Q$. Then $
\phi \in \Hom^\circ_A(L \oplus Q, P) \subset  \left[\begin{array}{cc}
\idm & \idm \\ R & R
\end{array} \right]
$
can be written in the form (\ref{E:FirstRed}). We distinguish the following cases.

\begin{enumerate}
\item[(i)] If $(a, b) \ne (0, 0)$, then there exists $\xi \in C$  such that 
$
\left( 
\begin{array}{cc}
a & b \\
c & d
\end{array}
\right) \xi = 
\left( 
\begin{array}{cc}
1 & 0 \\
\ast & \ast
\end{array}
\right).
$
Similarly to the previous cases, one can show that   $\phi \sim t^k \left(\begin{array}{cc} 1 & 0 \\ 0 & t^l \end{array}\right)$ for some $k \in \NN$ and $l \in \NN_0$. Up to automorphisms of the source and the target, $\phi$ admits a factorization
$$
L \oplus Q\xrightarrow{t^k \left(\begin{smallmatrix} 1 & 0 \\ 0 & t^l \end{smallmatrix}\right)} Q \oplus Q \subset P.
$$
It follows that the dimension vector of $\mathsf{Cok}(\phi)$ is $(2k+l, 2k +l)$. 
\item[(ii)] If $(a, b) =  (0, 0)$, then $(c, d) \ne   (0, 0)$. Hence, there exists $\xi \in C$  such that 
$
\left( 
\begin{array}{cc}
0 & 0 \\
c & d
\end{array}
\right) \xi = 
\left( 
\begin{array}{cc}
0 & 0 \\
0 & 1
\end{array}
\right).
$ It follows that $\phi \sim t^k \left(\begin{array}{cc} t^l & 0 \\ 0 & 1 \end{array}\right)$ for some $k \in \NN_0$ and $l \in \NN$. Up to automorphisms of the source and the target, $\phi$ admits a factorization
$$
L \oplus Q\xrightarrow{t^k \left(\begin{smallmatrix} t^l & 0 \\ 0 & 1 \end{smallmatrix}\right)} Q \oplus Q \subset P.
$$
It follows that the dimension vector of $\mathsf{Cok}(\phi)$ is again $(2k+l, 2k +l)$. 
\end{enumerate}
Note that the modules of types (i) and (ii) cannot be isomorphic, since their radicals are non-isomorphic. Indeed, for the first family, the radical is 
$\bigl(L \xrightarrow{t^{k+l}} Q\bigr) \oplus \bigl(Q \xrightarrow{t^{k}} Q\bigr)$, whereas for the second it is $\bigl(L \xrightarrow{t^{k}} Q\bigr) \oplus \bigl(Q \xrightarrow{t^{k+l}} Q\bigr)$.

\smallskip
\noindent
(d) Finally, consider the case $F \cong P$. Let 
\begin{equation}\label{E:Ansatz}
\phi = t^k 
\left( 
\begin{array}{cc}
a & b \\
c & d
\end{array}
\right) + t^{k+1} \left( 
\begin{array}{cc}
a' & b' \\
c' & d'
\end{array}
\right)
\in M_2(\idm) \cap \End_A^\circ(P)
\end{equation}
  for some $k \in \NN$, where $0 \ne \left( 
\begin{array}{cc}
a & b \\
c & d
\end{array}
\right) \in M_2(\RR)$ and $\left( 
\begin{array}{cc}
a' & b' \\
c' & d'
\end{array}
\right) \in M_2(R)$. Since $\left( 
\begin{array}{cc}
0 & 1 \\
-1 & 0
\end{array}
\right) \in C \subset \Aut_A(P)$, we may assume, without loss of generality, that $a \ne 0$. Moreover, there exists $\xi \in C$ such that $
\left( 
\begin{array}{cc}
a & b \\
c & d
\end{array}
\right) \xi = 
\left( 
\begin{array}{cc}
1 & 0 \\
\ast & \ast
\end{array}
\right).
$
The proof of the following result is straightforward. 

\smallskip
\noindent
\underline{Claim}. Let  $c,\lambda \in \RR$ and $d \in \RR^\ast$. There exist $\eta, \xi \in C$ such that 
\begin{equation}\label{E:Pseudodiag}
\eta \left( 
\begin{array}{cc}
1 & 0 \\
c & d
\end{array}
\right) \xi = 
\left( 
\begin{array}{cc}
1 & 0 \\
0 & \lambda
\end{array}
\right).
\end{equation}
if and only if $\lambda$ is a root of the polynomial
$
\pi_{c, d} = v^2 - \dfrac{d^2 + c^2 +1}{d} v + 1 \in \RR[v].
$
In particular, such $\lambda$ exists for every $c,d$. 

Hence, if $\phi \in M_2(\idm)$ given by (\ref{E:Ansatz}) is such that $\left|\begin{array}{cc}
a & b \\
c & d
\end{array} \right| \ne 0$, then $\phi \sim t^k \left( 
\begin{array}{cc}
1 & 0 \\
0 & \lambda
\end{array}
\right)$ for some $\lambda \in \RR^\ast$. Moreover, $t^k \left( 
\begin{array}{cc}
1 & 0 \\
0 & \lambda
\end{array}
\right) \sim t^{k'} \left( 
\begin{array}{cc}
1 & 0 \\
0 & \lambda'
\end{array}
\right)$ if and only if $k' = k$ and $\lambda' \in \left\{\lambda, \lambda^{-1} \right\}$.

We now consider the degenerate case when $\left|\begin{array}{cc}
a & b \\
c & d
\end{array} \right| =  0$. It is easy to see that there exist $\eta, \xi \in C$ such that 
$
\eta \left( 
\begin{array}{cc}
a & b \\
c & d
\end{array}
\right) \xi = 
\left( 
\begin{array}{cc}
1 & 0 \\
0 & 0
\end{array}
\right).
$ It follows that $\phi \sim t^k \left( 
\begin{array}{cc}
1 & 0 \\
0 & \lambda t^l 
\end{array}
\right)$ for some $\lambda \in \RR^\ast$ and $l \in \NN$. Hence, 
$\mathsf{det}(\phi) =  \mu t^{k+l}$ for some $\mu \in \RR^\ast$. Since $k$ appearing in (\ref{E:Ansatz}) is uniquely determined by the equivalence class of $\phi$, the second parameter $l \in \NN$ is uniquely determined by the class of $\phi$ as well. A straightforward computation shows that for $k, l \in \NN$ and $\lambda, \lambda' \in \RR^\ast$ the following statement holds: 
$$
t^k \left( 
\begin{array}{cc}
1 & 0 \\
0 & \lambda t^l 
\end{array}
\right) \sim t^k \left( 
\begin{array}{cc}
1 & 0 \\
0 & \lambda' t^l 
\end{array}
\right) \;  \mbox{if and only if} \; \lambda = \lambda'. 
$$
Now, let $\phi = t^k \left(\begin{smallmatrix} 1 & 0 \\ 0 & \lambda t^l\end{smallmatrix}\right)$ for some $k \in \NN, l \in \NN_0$ and $\lambda \in \RR^\ast$.  Then the morphism $P \stackrel{\phi}\lar P$ admits the following factorization: 
$$
P \subset L \oplus L \xrightarrow{t^k \left( 
\begin{smallmatrix}
1 & 0 \\
0 & \lambda t^l 
\end{smallmatrix}
\right)} Q \oplus Q = \mathsf{rad}(P) \subset P. 
$$
It follows that the dimension vector of $\mathsf{Cok}(\phi)$ is $(2k+l, 2k+l)$, as asserted. 
\end{proof}

\begin{remark}
The real Gelfand order $A$ is an example of a quadratic order in the sense of Iyama \cite{Iyama}. The methods of \cite{Iyama} yield, at least in principle, a classification of all indecomposable objects in $\Rep(A)$. Moreover, indecomposable objects of $\Rep(A)$ can also be studied via the classification of indecomposable representations of semi-linear clannish algebras due to Bennett-Tennenhaus and Crawley-Boevey \cite{BTCB}. However, neither Theorem \ref{T:Schurian} nor Theorem \ref{T:AbsCyclic} appears to follow from these approaches to the classification of indecomposable objects.
\end{remark}

\section{Some categorical constructions related with field extensions}\label{SC:CatConstr}
\begin{definition}
Let $\kk$ be a field, $\cC$ be a $\kk$-linear category and $\Gamma$ be a $\kk$-algebra. The category $\cC_\Gamma$ is defined as follows:
\begin{enumerate}
\item[(a)] $\mathsf{Ob}\bigl(\cC_\Gamma\bigr) = \mathsf{Ob}(\cC)$.
\item[(b)] $\Hom_{\cC_\Gamma}(X, Y) =  \Gamma \otimes_{\kk} \Hom_{\cC}(X, Y)$ for all $X, Y \in  \mathsf{Ob}\bigl(\cC_\Gamma\bigr) = \mathsf{Ob}(\cC)$.
\item[(c)] The composition of morphisms is induced by the composition of morphisms in $\cC$ and product in $\Gamma$. 
\end{enumerate}
\end{definition}

\noindent
Note that $\cC_\Gamma$ is again a  $\kk$-linear category and we have a  $\kk$-linear functor $\cC \lar \cC_\Gamma$, which is trivial on objects and which sends  each morphism $f$ to $\mathbbm{1} \otimes f$. 

For any additive category $\cD$ we denote by $\cD^\omega$ its idempotent completion. Recall that the objects of $\cD^\omega$ are pairs $(X, e)$, where $X \in \mathsf{Ob}(\cD)$ and $e \in \End_{\cD}(X)$ an idempotent. For $(X, e), (Y, f) \in \mathsf{Ob}(\cD^\omega)$ we have: 
$$
\Hom_{\cD^\omega}\bigl((X, e), (Y, f)\bigr) = f \Hom_{\cD}(X, Y) e.
$$
We have a functor $\cD \stackrel{\mathbbm{K}}\lar \cD^\omega, \; X \mapsto (X, \mathbbm{1}_X)$ inducing for any idempotent complete additive category $\cE$ an equivalence of categories of additive functors
\begin{equation}\label{E:UnivPropIdC}
\mathsf{Fun}_{\mathsf{add}}(\cD^\omega, \cE) \stackrel{\mathbbm{K}^\sharp}\lar \mathsf{Fun}_{\mathsf{add}}(\cD, \cE).
\end{equation}
See also \cite[Section 6]{Buehler} for a detailed treatment of constructions related to the idempotent completion of an additive category.
\begin{lemma}\label{L:IndMatrixAlg} Let $\cC$ be an additive $\kk$-linear category. 
For any $n \in \NN$, let $\cC_n:= \cC_{M_n(\kk)}$ and $\widetilde{\cC}_n = \cC_{n}^\omega$. Then there exists a $\kk$-linear equivalence of categories $\cC^\omega \xrightarrow{\mathbbm{I}} \widetilde{\cC}_n$. In particular, if $\cC$ is idempotent complete then $\cC \simeq \widetilde{\cC}_n$ for any $n \in \NN$. 
\end{lemma}

\begin{proof} First note that for any $X \in \mathsf{Ob}\bigl(\cC_n\bigr) = \mathsf{Ob}(\cC)$ we have isomorphisms of $\kk$-algebras:
$$
\End_{\cC_n}(X) = M_n(\kk) \otimes_{\kk} \End_{\cC}(X) \cong M_n\bigl(\End_{\cC}(X)\bigr).
$$ 
For any $1 \le k \le n$ we put: $u^{(k)}_X = e_{kk} \otimes \mathbbm{1}_X$, where $e_{kk} \in M_n(\kk)$ is the $k$-th standard primitive idempotent. It is clear that $\bigl(u_X^{(k)}\bigr)_{1 \le k \le n}$ 
is a family of orthogonal idempotents in $\End_{\cC_n}(X)$. We put: $u_X:= u^{(1)}_X$. Then $(X, u_X)$ is an object of the idempotent completion $\widetilde{\cC}_n$.
Next, for any $X, Y \in \mathsf{Ob}(\cC)$, we have natural isomorphisms of vector spaces over $\kk$: 
$$
\Hom_{\widetilde{\cC}_n}\bigl((X, u_X), (Y, u_Y)\bigr) = 
u_Y  \Hom_{{\cC}_n}(X, Y) u_X \cong \Hom_{\cC}(X, Y).
$$
As a consequence, we get a fully faithful $\kk$-linear functor $\cC \xrightarrow{\mathbbm{I}} \widetilde{\cC}_n, X \mapsto (X, u_X)$. Since for any $1 \le k, l \le n$ the idempotents $u_X^{(k)}$ and 
$u_X^{(l)}$ are conjugate, we have: 
$$
(X, \mathbbm{1}_X) \cong \bigoplus\limits_{k = 1}^n (X, u^{(k)}_X) \cong (X, u_X)^{\oplus n}.
$$
Since any  object of $\widetilde{\cC}_n$ has the form $(X, e)$, where $X \in \mathsf{Ob}\bigl(\cC_n\bigr)$ and
$e \in \End_{\cC_n}(X)$ is an idempotent, it is a direct summand of $(X, \mathbbm{1}_X) \cong 
\mathbbm{I}(X^{\oplus n})$. The universal property of the idempotent completion (\ref{E:UnivPropIdC}) implies that we have a fully faithful functor $\cC^\omega \xrightarrow{\mathbbm{I}} \widetilde{\cC}_n$, which is moreover an equivalence of categories. 
\end{proof}

The following result is, of course, well-known and we provide its proof for a convenience of the reader. 
\begin{lemma}\label{L:ExtScalars}
Let $\Gamma$ be a $\kk$-algebra and $\kk \subseteq \kk'$ be a finite field extension. Then for any 
$M, N \in \Gamma\mbox{--}\mathsf{Mod}$ the canonical map
\begin{equation}\label{E:ExtCan}
\kk' \otimes_{\kk} \Hom_{\Gamma}(M, N)  \lar \Hom_{\kk' \otimes_{\kk} \Gamma}(\kk' \otimes_{\kk} M, \kk' \otimes_{\kk} N), \; a \otimes f \mapsto \bigl(b \otimes x \mapsto ab \otimes f(x)\bigr)
\end{equation}
is an isomorphism of $\kk'$-vector spaces. 
\end{lemma}

\begin{proof}  For any $N \in \Gamma\mbox{--}\mathsf{Mod}$ we have a natural transformation of left exact contravariant functors 
$$
 \kk' \otimes_{\kk} \Hom_{\Gamma}(\,-\,, N)   \stackrel{\vartheta}\lar \Hom_{\kk' \otimes_{\kk} \Gamma}(\kk' \otimes_{\kk}\,-\,, \kk' \otimes_{\kk} N)
$$
from the category $\Gamma\mbox{--}\mathsf{Mod}$ to the category of vector spaces over $\kk'$, given by the family of maps (\ref{E:ExtCan}). It is clear that for  the regular module $\Gamma = \, {_{\Gamma}\Gamma}$, the corresponding map $\theta_{\Gamma}$ is an isomorphism. It follows that  $\vartheta_{F}$ is an isomorphism for any free $\Gamma$-module $F$. Taking free presentations, we conclude that $\vartheta_M$ is an isomorphism for any $M \in \Gamma\mbox{--}\mathsf{Mod}$.
\end{proof}

\begin{lemma}\label{L:ExtScalars2}
Let $\Gamma$ be a $\kk$-algebra, $\kk \subseteq \kk'$ be a finite field extension and $\Gamma' = \kk' \otimes_{\kk} \Gamma$. Then we have a fully faithful functor 
$$
\bigl(\Gamma\mbox{--}\mathsf{Mod}\bigr)_{\kk'} \stackrel{\mathbbm{J}}\lar \Gamma'\mbox{--}\mathsf{Mod}, \; M \longmapsto \kk' \otimes_{\kk} M. 
$$
Moreover, if $\kk \subseteq \kk'$ is separable then the induced functor 
$$
\bigl(\Gamma\mbox{--}\mathsf{Mod}\bigr)_{\kk'}^{\omega} \stackrel{\widetilde{\mathbbm{J}}}\lar \Gamma'\mbox{--}\mathsf{Mod}
$$
is an equivalence of categories. 
\end{lemma}

\begin{proof}
The fact that the functor $\mathbbm{J}$ is fully faithful, follows from Lemma \ref{L:ExtScalars}.

If $\kk \subseteq \kk'$ is a finite separable field extension then the product map $\kk' \otimes_{\kk} \kk' \stackrel{\mu_\circ}\lar  \kk'$ splits as a morphism of $\kk'$-bimodules. Hence, we have an element
$w_\circ \in \kk' \otimes_{\kk} \kk'$ such that $\mu_\circ(w_\circ) = 1$ and $a w_\circ = w_\circ a$ for all $a \in \kk'$. Next, consider the canonical $\kk'$-linear map $\kk' \otimes_{\kk} \kk' \stackrel{\kappa}\lar \Gamma' \otimes_{\Gamma} \Gamma'$ and put $w = \kappa(w_\circ)$. Then $a w = w a$ for all 
$a \in \Gamma'$ and $\mu(w) = 1$, where $\Gamma' \otimes_{\Gamma} \Gamma' \stackrel{\mu}\lar  \Gamma'$ is the multiplication map. It follows that $\mu$ splits as a morphism of $\Gamma'$-bimodules.

For any $X \in \Gamma'\mbox{--}\mathsf{Mod}$ we have a split epimorphism of $\Gamma'$-modules $\kk' \otimes_{\kk} X \cong \Gamma' \otimes_\Gamma  X \rightarrowdbl X$ given by the composition
$$
\Gamma' \otimes_{\Gamma} X \cong \bigl(\Gamma' \otimes_{\Gamma} \Gamma'\bigr) \otimes_{\Gamma'} X \xrightarrow{\mu \otimes \mathbbm{1}} \Gamma' \otimes_{\Gamma'}  X \cong X. 
$$
It follows that any $X \in \Gamma'\mbox{--}\mathsf{Mod}$ is a direct summand of $\mathbbm{J}(X)$, implying that the functor $\widetilde{\mathbbm{J}}$ is essentially surjective. Since it is fully faithful, it is an equivalence of categories.  
\end{proof}

 \begin{remark} Of course, we also get a fully faithful functor $\bigl(\Gamma\mbox{--}\mathsf{mod}\bigr)_{\kk'} \stackrel{\mathbbm{J}}\lar \Gamma'\mbox{--}\mathsf{mod}, \; 
 M \mapsto \kk' \otimes_{\kk} M$ between the corresponding categories of finite dimensional modules. If $\kk \subseteq \kk'$ is a finite separable extension then $\bigl(\Gamma\mbox{--}\mathsf{mod}\bigr)_{\kk'}^{\omega} \stackrel{\widetilde{\mathbbm{J}}}\lar \Gamma'\mbox{--}\mathsf{mod}
$
is an equivalence of categories. 
\end{remark}

\smallskip
\noindent
In what follows, \(G\) is a finite group whose underlying set is \(\bigl\{g_1, \dots, g_t\bigr\}\), where \(g_1 = e\) is the identity element.

\begin{definition}\label{D:GroupAction}
Let $\Gamma$ be a $\kk$-algebra and $G \stackrel{\phi}\lar \Aut_{\kk}(\Gamma)$ be a group homomorphism. 
 The corresponding crossed product  $\Gamma\bigl[G, \phi\bigr]$ is a free left $\Gamma$--module of rank $ t = |G|$: 
\begin{equation}
\Gamma\bigl[G, \phi] = \Bigl\{\sum\limits_{g \in G} a_g[g]  \, \big| a_g \in \Gamma  \Bigr\}
\end{equation}
equipped with the product given by the rule
\begin{equation}\label{E:SkewProductRule}
a[f] \cdot b[g] := a \phi_f(b) [fg] \; \mbox{\rm for any} \; a, b \in  
\Gamma \; \mbox{\rm and} \;  f, g \in G.
\end{equation}
Here and further we write $\phi_f$ instead of $\phi(f)$ for $f \in G$.
It is not difficult to show that $\Gamma\bigl[G, \phi\bigr]$ is again a $\kk$-algebra,  whose multiplicative unit element is $1[e]$.
\end{definition}

\begin{remark} In general, the definition of the crossed product $\Gamma[G, (\phi, \omega)]$
involves an appropriate  two-cocycle 
$G \times G \stackrel{\omega}\lar \Gamma^\ast$, where $\Gamma^\ast$ is the group of units of $\Gamma$; see \cite{ReitenRiedtmann}. In our setting,  $\omega$ is trivial and  $\Gamma\bigl[G, (\phi, \omega)\bigr] = \Gamma[G, \phi]$.
\end{remark}

\begin{definition} Let $\cC$ be an $\kk$-linear category and $G $ be a finite group. We say that $G$ acts on $\cC$ if for any $g \in G$ we have a $\kk$-linear auto-equivalence $\Phi_g: \cC \lar \cC$ such that
\begin{enumerate}
\item[(a)] $\Phi_g \Phi_h = \Phi_{gh}$ for all $g, h \in G$.
\item[(b)] $\Phi_e = \mathsf{Id}$. 
\end{enumerate}
Following \cite{DFO}, we  define a new category $\cC[G, \Phi]$ as follows: 
\begin{enumerate}
\item[(a)] $\mathsf{Ob}\bigl(\cC[G, \Phi]\bigr) = \mathsf{Ob}\bigl(\cC\bigr)$. 
\item[(b)] For any $X, Y \in \mathsf{Ob}(\cC)$ we put: 
$$
\Hom_{\cC[G, \Phi]}(X, Y) = \bigoplus\limits_{g \in G} \Hom_{\cC}\bigl(X, \Phi_g(Y)\bigr),
$$
i.e.~a morphism $\xi \in \Hom_{\cC[G, \Phi]}(X, Y)$ is given by a family  $\bigl(X \stackrel{{\xi}_g}\lar \Phi_g(Y)\bigr)_{g \in G}$ of morphisms in $\cC$.
\item[(c)] The composition of morphisms in $\cC[G, \Phi]$ is given by the following rule: 
$$
\bigl(\eta_{g_1}, \dots, \eta_{g_t}\bigr) \circ \bigl(\xi_{g_1}, \dots, \xi_{g_t}\bigr) = \Bigl(\dots, \sum\limits_{f \in G} \Phi_f\bigl(\eta_{f^{-1} g}\bigr) \circ \xi_f, \dots\Bigr)_{g \in G}.
$$
\end{enumerate}
\end{definition}

\begin{remark} In general, the definition of an  action of a finite group $G$ on a category $\cC$ involves a choice of a family of compatible isomorphisms of functors $\bigl(\Phi_g \Phi_h \xrightarrow{\omega_{g, h}} \Phi_{gh}\bigr)_{g, h \in G}$; see \cite{DFO}. However, this general setting is not used in this work. 
\end{remark}

\begin{remark}\label{R:SubtletiyGroupAct} Assume that a finite group $G$ acts on a $\kk$-linear additive category $\cC$. Then $G$ also acts on the idempotent completion $\cC^\omega$ and we get a fully faithful functor $\cC[G, \Phi] \xrightarrow{\mathbbm{K}} \cC^\omega[G, \Phi], \; X \mapsto (X, \mathbbm{1}_X)$, which induces a $\kk$-linear equivalence of idempotent completions
\begin{equation}\label{E:subtle}
\bigl(\cC[G, \Phi]\bigr)^\omega \xrightarrow{\widetilde{\mathbbm{K}}} \bigl(\cC^\omega[G, \Phi]\bigr)^\omega.
\end{equation}
\end{remark}

\begin{example}
Let $\kk \subseteq \kk'$ be a finite field  extension and $G = \mathsf{Aut}_{\kk}(\kk')$ be the corresponding automorphism group. For any $g \in G$ let $\kk' \stackrel{\phi_g}\lar \kk'$ be the corresponding $\kk$-linear automorphism.
Then for any $\kk$-linear category $\cC$,  the group $G$ acts $\kk$-linearly on the category $\cC_{\kk'}$:
\begin{enumerate}
\item[(a)] For any $g \in G$ and $X \in \mathsf{Ob}(\cC_{\kk'}) = \mathsf{Ob}(\cC)$ we put: $\Phi_g(X) = X$.
\item[(b)] For any $X, Y \in \mathsf{Ob}(\cC_{\kk'})$ we have by definition: $\Hom_{\cC_{\kk'}}(X, Y) = 
\kk' \otimes_{\kk} \Hom_{\cC}(X, Y)$. Hence, we put:  $\Phi_g(a \otimes \xi) = \phi_g(a) \otimes \xi$ for any $\xi \in \Hom_{\cC}(X, Y)$ and $a \in \kk'$.
\end{enumerate}
\end{example}

\begin{example}
Let $\Gamma$ be a $\kk$-algebra and $G \stackrel{\phi}\lar \Aut_{\kk}(\Gamma)$ be a group homomorphism.
For any left $\Gamma$-module $M$ and $g \in G$ we have a new $\Gamma$-module $_{g}M$ defined as follows:
\begin{enumerate}
\item[(a)] $_{g}M = M$ as a set and a vector space over $\kk$.
\item[(b)] The $\Gamma$-module structure on $_{g}M$ is defined as follows:  for any $a \in \Gamma$ and $x \in M$ we put: $a \ast x = \phi_{g^{-1}}(a) \circ x$, where $\circ$ is the action of $\Gamma$ on $M$. 
\end{enumerate} 
For any $g \in G$ we have a functor $\Gamma\mbox{--}\mathsf{Mod} \xrightarrow{\Phi_g} \Gamma\mbox{--}\mathsf{Mod}$ sending each left $\Gamma$-module $M$ to $_{g}M$ and acting identically  on morphisms. One can check that  $\Phi_g \Phi_h = \Phi_{gh}$ for all $g, h \in G$.
\end{example}

\begin{lemma} Let $\cC$ be a $\kk$-linear category, $\kk \subseteq \kk'$ be a finite field extension and 
$G = \mathsf{Aut}_{\kk}(\kk')$ be the corresponding automorphism group. Then the categories $\cC_{\kk'}[G, \Phi]$ and $\cC_{\kk'[G, \phi]}$ are isomorphic. 
\end{lemma}
\begin{proof}
 Let $|G|=t$.
By definition, we have: $\mathsf{Ob}\bigl(\cC_{\kk'}[G, \Phi]\bigr) = \mathsf{Ob}\bigl(\cC_{\kk'[G, \phi]}\bigr) = \mathsf{Ob}(\cC)$. Let $X, Y \in \mathsf{Ob}(\cC)$. Then a morphism in $\Hom_{\cC_{\kk'}[G,\Phi]}(X, Y)$ is a tuple $\xi = \bigl(\sum\limits_{i = 1}^{n_k} a_i^{(k)} \otimes \xi_i^{(k)}\bigr)_{1 \le k \le t}$, where 
$\xi_{i}^{(k)} \in \Hom_{\cC}(X, \Phi_{g_k}Y)$ and $a_i^{(k)} \in \kk'$ for all $1\le k \le t$ and $1 \le i \le n_k$. We define 
$\cC_{\kk'}[G,\Phi] \xrightarrow{\EE} \cC_{\kk'[G, \phi]}$ to be the identity on objects, whereas a  morphism $\xi$ as above is sent to  $\sum\limits_{k = 1}^t \sum\limits_{i = 1}^{n_k}  
 a_i^{(k)} \bigl[g_k\bigr]\otimes \xi_i^{(k)}$. It is easy to see that $\EE$ is an isomorphism of categories. 
\end{proof}

The following theorem is due to \cite{DFO}. We provide its proof for the reader's convenience. 

\begin{theorem}\label{T:DFO} Let $\Gamma$ be a $\kk$-algebra, $G$ be a finite group, 
$G \stackrel{\phi}\lar \Aut_{\kk}(\Gamma)$ be a group homomorphism and $\Lambda = \Gamma[G, \phi]$.  
Then we have a fully faithful $\kk$-linear functor 
\begin{equation}
\Gamma\mbox{--}\mathsf{Mod}[G, \Phi] \xrightarrow{\EE_G} \Lambda\mbox{--}\mathsf{Mod}.
\end{equation}
sending an object $M$ to $\Lambda\otimes_{\Gamma} M$. Moreover, if $t = |G|$ is invertible in $\kk$ then $\EE_G$ induces an equivalence of categories
\begin{equation}
\Bigl(\Gamma\mbox{--}\mathsf{Mod}[G, \Phi]\Bigr)^\omega \xrightarrow{\widetilde{\EE}_G} 
\Lambda\mbox{--}\mathsf{Mod}.
\end{equation}
\end{theorem}
\begin{proof} Let $M$ and $N$ be arbitrary $\Gamma$-modules. Then we have a natural isomorphism of 
$\Gamma$-modules $\Lambda\otimes_ {\Gamma} N \cong \oplus_{g \in G} \;  {_{g}N}$. 
Next, 
we have an isomorphism of $\kk$-vector spaces:
$$
\Hom_{\Lambda}\bigl(\Lambda\otimes_ {\Gamma} M, \Lambda\otimes_ {\Gamma} N\bigr) \cong \Hom_{\Gamma}(M, \Lambda\otimes_ {\Gamma} N\Bigr) \cong \bigoplus\limits_{g \in G} \Hom_{\Gamma}\bigl(M,\,  _{g}N\bigr),
$$
natural both in $M$ and $N$. Hence, any morphism $\xi \in \Hom_{\Lambda}\bigl(\Lambda\otimes_ {\Gamma} M, \Lambda\otimes_ {\Gamma} N\bigr)$ is uniquely determined  by a family of morphisms of $\Gamma$-modules $\bigl(M \xrightarrow{\xi_g} {_{g}N}\bigr)_{g \in G}$. This shows that the assignment 
$M \mapsto \Lambda\otimes_{\Gamma} M$ extends to a fully faithful functor 
$\Gamma\mbox{--}\mathsf{Mod}[G, \Phi] \xrightarrow{\EE_G} \Lambda\mbox{--}\mathsf{Mod}$.

\smallskip
\noindent
Next, note that the element
$$
w := \sum\limits_{g \in G} \bigl[g^{-1}\bigr] \otimes [g] \;\in \;  \Lambda \otimes_{\Gamma}
\Lambda
$$
is central. Hence, under the assumption that $t$ is invertible in $\kk$, 
the morphism of $\Lambda$-bimodules 
$
\Lambda  \otimes_{\Gamma} \Lambda  \stackrel{\mu}\lar  \Lambda 
$
given by the multiplication map admits a splitting $\Lambda  \stackrel{\sigma}\lar  
\Lambda  \otimes_{\Gamma} \Lambda$ defined by the assignment $1 \mapsto \frac{1}{t} w$.

For any $X \in \Lambda\mbox{--}\mathsf{Mod}$ we have a split epimorphism of $\Lambda$-modules $\Lambda \otimes_{\Gamma}X \rightarrowdbl X$ given by the composition
$$
\Lambda \otimes_{\Gamma} X \cong \bigl(\Lambda \otimes_{\Gamma} \Lambda\bigr) \otimes_{\Lambda} X \xrightarrow{\mu \otimes \mathbbm{1}} \Lambda \otimes_\Lambda X \cong X. 
$$
It follows that any $X \in \Lambda\mbox{--}\mathsf{Mod}$ is a direct summand of $\mathbbm{E}_G(X)$, implying that the functor $\Bigl(\Gamma\mbox{--}\mathsf{Mod}[G, \Phi]\Bigr)^\omega \xrightarrow{\widetilde{\EE}_G} 
\Lambda\mbox{--}\mathsf{Mod}$ is an equivalence of categories. 
\end{proof}

The following result is well-known, see e.g.~\cite[Theorem 5.6.6]{DrozdKirichenko}.
\begin{proposition}\label{P:GaloisGroup}
Let $\kk \subseteq \kk'$ be a finite Galois extension and $G = \mathsf{Aut}_{\kk}(\kk')$ be the corresponding Galois group. Then we have: 
$
\kk'[G, \phi] \cong M_t(\kk).
$
\end{proposition}

\begin{corollary}\label{C:GaloisExt}
Let $\cC$ be a $\kk$-linear additive category and $\kk \subseteq \kk'$ be a finite Galois extension with the Galois group $G$. Then we have a $\kk$-linear equivalence of categories
$$
\cC^\omega \simeq  \bigl(\cC_{\kk'}[G, \Phi]\bigr)^\omega.
$$
\end{corollary}
\begin{proof}
This is a consequence of Lemma \ref{L:IndMatrixAlg} and Proposition \ref{P:GaloisGroup}.
\end{proof}

\section{Principal block of the category of real Harish-Chandra modules for the group $\mathsf{SL}_2(\RR)$}\label{S:HCModules}

\noindent
Let $\lieg = \mathfrak{sl}_2(\RR)$, $K = \SO_2(\RR)$ and $U(\lieg)$ be the universal enveloping algebra of $\lieg$. We denote
$$
h_\circ = \left(
\begin{array}{cc}
1 & 0 \\
0 & -1
\end{array}
\right), \; 
x_\circ = \left(
\begin{array}{cc}
0 & 1 \\
0 & 0
\end{array}
\right),\;
y_\circ = \left(
\begin{array}{cc}
0 & 0 \\
1 & 0
\end{array}
\right).
$$

It is well-known that any continuous finite dimensional indecomposable representation of $K$ over $\RR$ is automatically irreducible and is either
\begin{enumerate}
\item[(a)] trivial $V_0 = \RR$, or
\item[(b)] two-dimensional: $V_n = \RR^2$ with the action of $k \in K$ given by  $v \mapsto k^n \cdot v$ for $n \in \NN$, where $\cdot$ denotes the action of the fundamental representation of $K$.
\end{enumerate}

\begin{definition} A real vector space $M$ is a Harish-Chandra $(\lieg, K)$-module if it has a structure $(M, \circ)$ of a finitely generated $U(\lieg)$-module as well as a structure $(M, \,\cdot\,)$ of an admissible representation of $K$ (meaning that $M \cong \oplus_{n \in \NN_0} V_n^{\oplus m_n}$ is a direct sum of finite dimensional continuous representations  of $K$ with finite multiplicities $m_n \in \NN_0$) such that 
$$
\left.\frac{d}{dt}\right|_{t = 0} \exp(t z)\cdot  v  = z \circ v 
$$
and
$$
  \bigl(\Ad_k(a)\bigr) \circ v
=
  k \cdot \bigl(a \circ (k^{-1} \cdot v)\bigr)
$$
for any~$v \in M$, $z \in  \mathfrak{k} = \mathsf{Lie}(K)$,  $k \in K$  and~$a \in \lieg$, where~$\Ad$ is the adjoint action of~$\SL_2(\RR)$ on~$\lieg$. The corresponding category of Harish-Chandra modules is denoted by $\HC(\lieg, K)$. 
\end{definition}
Let $c = h_\circ^2 - 2h_\circ + 4x_\circ y_\circ  = h_\circ^2 + 2h_\circ + 4y_\circ x_\circ \in U({\lieg})$ be the Casimir element of $\lieg$. Then $\RR[c]$ is the center of $U(\lieg)$. 
The category $\HC(\lieg, K)$ splits into a coproduct of blocks:
\begin{equation}\label{E:blocks}
\HC(\lieg, K) = \coprod_{\gamma  \in \Spec_\circ (\RR[c])} \HC_{\gamma}(\lieg, K),
\end{equation}
where $\Spec_\circ\bigl(\RR[c]\bigr)$ is the set of non-zero prime ideals in $\RR[c]$ and $\HC_{\gamma}(\lieg, K)$
is the subcategory of $\HC(\lieg, K)$ consisting of such modules $M$ that $\ga^mM=0$ for some $m$ (depending on $M)$. 
Of major interest is the so-called \emph{principal block}
$
\HC_{\circ}(\lieg, K) 
$
whose objects are  those objects $M$ of $\HC(\lieg, K)$ for which 
there exists $m \in \NN$ (depending on $M$) such that $c^m M = 0$.

Let $\tilde{\lieg} = \mathfrak{sl}_2(\CC)$ be the complexification of $\lieg$
and 
$$
h = \left(
\begin{array}{cc}
0 & -i \\
i & 0
\end{array}
\right), \; 
x = \frac{1}{2}\left(
\begin{array}{cc}
1 & i \\
i & -1
\end{array}
\right),\;
y = \frac{1}{2}\left(
\begin{array}{cc}
1 & -i \\
-i & -1
\end{array}
\right).
$$
Then $\tilde{\lieg} = \bigl\langle h, x, y\bigr\rangle_\CC$ and 
\begin{equation}\label{E:Standard}
[h, x] = 2x, [h, y] = -2y, [x, y] = h.
\end{equation}
Moreover, $c = h^2 - 2h + 4x y  = h^2 + 2h + 4y x$ is another expression for the Casimir element introduced above. 
Let $\tilde{K}=\SO_2(\CC)$ be the complexification of $K$, i.e.
$$
\tilde{K} = 
\left\{k = \left.\left(\begin{array}{cc} \alpha & - \beta \\ \beta & \alpha \end{array}\right) \right| \alpha, \beta \in \CC, \mathsf{det}(k) = 1\right\}.
$$

\begin{remark}
Let $M$ be a left $U(\tilde{\lieg})$-module. For any $n\in \ZZ$ we denote $M_n := \bigl\{v \in M \,|\, h \circ v = nv \bigr\}$. Then $M$ is a Harish-Chandra $(\tilde{\lieg}, \tilde{K})$-module if and only if
\begin{enumerate}
\item[(a)] $M$ is finitely generated over $U(\tilde{\lieg})$.
\item[(b)] $M \cong \oplus_{n \in \mathbb{Z}} M_n$ as a complex vector space and $\dim_{\CC}(M_n) < \infty$ for any $n \in \ZZ$. 
\end{enumerate}
\end{remark}

\smallskip
\noindent
Let $M$ be an object of $\HC(\tilde\lieg, \tilde{K})$.
It follows from (\ref{E:Standard}) that $x(M_{l-1}) \subset M_{l+1}$ and 
$y(M_{l+1}) \subset M_{l-1}$ for any $l \in \ZZ$.

Let $\HC_\circ(\tilde{\lieg}, \tilde{K})$ be the full subcategory of 
$\HC(\tilde{\lieg}, \tilde{K})$ consisting of those modules $M$ for which 
there exists $m \in \NN$ (depending on $M$) such that $c^m M = 0$. 

\begin{proposition}\label{T:blocks}
The category $\HC_\circ(\tilde{\lieg}, \tilde{K})$ is equivalent to the category $\RepQ(O)$ of finite dimensional nilpotent representations of the Gelfand quiver (\ref{E:GelfandQuiver}). 
\end{proposition}

\begin{proof}
Consider $M \in \HC_{\circ}(\tilde{\lieg}, \tilde{K})$. Then the $\tilde{\lieg}$-module structure on $M$ is determined by the following diagram 

\begin{equation}\label{E:DiagramHCmodule}
\xymatrix{
 \dots  M_{-4} \;  \ar@/^/[r]^{x}
&  \ar@/^/[l]^{y} M_{-2} \ar@/^/[r]^{x} & M_{0} \ar@/^/[l]^{y}  \ar@/^/[r]^{x} & M_2 \ar@/^/[l]^{y}  \ar@/^/[r]^{x}&  \ar@/^/[l]^{y}   M_4 \; \dots 
}
\end{equation}
We have a functor 
\begin{equation}
\HC_{\circ}(\tilde{\lieg}, \tilde{K}) \xrightarrow{\II} \RepQ(O),
\end{equation}
assigning  to a Harish-Chandra module $M$ (given by the diagram (\ref{E:DiagramHCmodule}))  the following representation
\begin{equation}\label{E:HCtoRepsgelfand}
\xymatrix{
M_{-2} \ar@/^/[r]^{x}
& \ar@/_/[r]_{x} \ar@/^/[l]^{y} M_0    &  
M_2 \ar@/_/[l]_{y}
}
\end{equation}
of the Gelfand quiver (\ref{E:GelfandQuiver}). It is well-known and not difficult to check  that $\II$  is an equivalence of categories.
\end{proof}

\begin{remark}
Note that $O$ is isomorphic to the arrow ideal completion of the path algebra of the Gelfand quiver (\ref{E:GelfandQuiver}). This isomorphism is given by  the following identifications:
\begin{equation}\label{E:IdentifGelfandQuiver}
\left\{
\begin{array}{l}
\varepsilon_+ = 
\left(
\begin{array}{ccc}
1 & 0 & 0 \\
0 & 0 & 0 \\
0 & 0 & 0
\end{array}
\right), \; 
\varepsilon_- = 
\left(
\begin{array}{ccc}
0 & 0 & 0 \\
0 & 1 & 0 \\
0 & 0 & 0
\end{array}
\right), \; \varepsilon_\star = 
\left(
\begin{array}{ccc}
0 & 0 & 0 \\
0 & 0 & 0 \\
0 & 0 & 1
\end{array}
\right)
\\
a_+= \left(
\begin{array}{ccc}
0 & 0 & 0 \\
0  & 0 & 0 \\
t & 0 & 0
\end{array}
\right), \;
a_- = \left(
\begin{array}{ccc}
0 & 0 & 0 \\
0 & 0 & 0 \\
0 & t & 0
\end{array}
\right),
\\
b_+= \left(
\begin{array}{ccc}
0 & 0 & 1 \\
0 & 0 & 0 \\
0 & 0 & 0
\end{array}
\right), \;
b_- = \left(
\begin{array}{ccc}
0 & 0 & 0 \\
0 & 0 & 1 \\
0 & 0 & 0
\end{array}
\right).
\end{array}
\right.
\end{equation}
In these terms, we get an equivalence of categories
\begin{equation}\label{E:GelfandEq}
\RepQ(O) \stackrel{\EE}\lar \Rep(O).
\end{equation}
Next, the $\CC$-algebra $O$  has an $\RR$-linear involution $O \stackrel{\sigma}\lar O$, given by the formula
\begin{equation}\label{E:Involution}
\sigma(\varepsilon_\star) = \varepsilon_\star, \; \sigma(\varepsilon_\pm) = \varepsilon_\mp, \; 
\sigma(a_\pm) = a_\mp, 
\sigma(b_\pm) = b_\mp \;  \mbox{\rm and} \; \sigma(\lambda) = \bar\lambda \; \mbox{\rm for all}\; \lambda \in \CC.
\end{equation} 
Hence, it induces an involutive $\RR$-linear auto-equivalence 
$
\Rep(O) \stackrel{\sigma^\sharp}\lar \Rep(O).
$
It is easy to see that under (\ref{E:GelfandEq}), it corresponds to the $\RR$-linear equivalence $\RepQ(O) \stackrel{(-)^\ddagger}\lar \RepQ(O)$ given by the rule
\begin{equation}\label{E:ComplexConj}
\xymatrix
{
V_- \ar@/^/[rr]^{A_{-}}  & &  V_\star \ar@/^/[ll]^{B_{-}}
 \ar@/_/[rr]_{B_{+}}
 & &
\ar@/_/[ll]_{A_{+}} V_+
}
 \quad \stackrel{(-)^\ddagger}\longmapsto \quad
 \xymatrix
{
V_+ \ar@/^/[rr]^{A_+^\ddagger}  & &  V_\star \ar@/^/[ll]^{B_{+}^\ddagger}
 \ar@/_/[rr]_{B_{-}^\ddagger}
 & &
\ar@/_/[ll]_{A_{-}^\ddagger} V_-
}
\end{equation}
where $A_\pm^\ddagger$ (resp. $B_\pm^\ddagger$) is the complex conjugate map of $A_\pm$ (resp. $B_\pm$). 
\end{remark}

We have an involutive $\RR$-linear auto-equivalence $\HC(\tilde{\lieg}, \tilde{K}) \stackrel{(-)^\dagger}\lar \HC(\tilde{\lieg}, \tilde{K})$ given by the complex conjugation of the $\CC$-vector space structure on a Harish-Chandra  $(\tilde{\lieg}, \tilde{K})$-module. 
The proof of the following statement is straightforward. 
\begin{lemma}\label{L:ComplexConj} The following diagram of categories and functors is commutative:
\begin{equation}
\begin{array}{c}
 \begin{xy}
\xymatrix{
 \HC_\circ(\tilde{\lieg}, \tilde{K})  \ar[rr]^{\II} \ar[d]_-{(-)^\dagger}
&& \RepQ(O) \ar[d]^-{(-)^\ddagger} \ar[rr]^{\EE}  & & \Rep(O) \ar[d]^-{\sigma^\sharp}\\
 \HC_\circ(\tilde{\lieg}, \tilde{K})  \ar[rr]^{\II} && \RepQ(O) \ar[rr]^{\EE}  & & \Rep(O).
}
\end{xy}
\end{array}
\end{equation}
\end{lemma}

\smallskip
\noindent
Let $G = \mathsf{Gal}(\CC/\RR) = \langle \sigma | \sigma^2 = e\rangle$. Then we have a group homomorphism
$G \stackrel{\phi}\lar  \Aut_{\RR}(O)$ assigning to $\sigma$ the automorphism $\sigma$ given by (\ref{E:Involution}). It follows from \cite[Theorem 4.1]{BurbanDrozdSurvey} that the crossed product $ O[G, \phi]$ is again a real nodal order. 

\begin{lemma}\label{L:MoritaEq}
The order $B:= O[G, \phi]$ is Morita equivalent to the real Gelfand order $A$. 
\end{lemma}

\begin{proof} Let $\varepsilon_\pm$ and $\varepsilon_\star$ be the primitive idempotents of $O$ given by (\ref{E:IdentifGelfandQuiver}). Then $e_\star^{\pm} = \dfrac{1 \pm i[\sigma]}{2} \varepsilon_\star$ is a pair of orthogonal idempotents in $B$ and $e_\star^+ + e_\star^- = \varepsilon_\star$. It follows from   (\ref{E:SkewProductRule}) that 
$[\sigma] e_\star^\pm = e_\star^\mp [\sigma]$, hence the idempotents $e_{\star}^+$ and $e_{\star}^-$ are conjugate in $B$. Let $\varepsilon = e_{\star}^+ + \varepsilon_+$ and $B'=\varepsilon B \varepsilon$. Then 
$B$ and $B'$ are Morita equivalent. A straightforward computation shows that 
$B' \cong A$ as $\RR$-algebras. 
\end{proof}

\begin{remark} The following proof of Lemma~\ref{L:MoritaEq} was communicated to us by Bill Crawley-Boevey. It follows from \cite[Proposition 1.2]{ARS} that $O[G, \phi]$ is Morita equivalent to the $\RR$-algebra of invariants $O^G$. We have the following invariant elements:
$$
e = \varepsilon_+ + \varepsilon_-, \, \jmath = i(\varepsilon_+ - \varepsilon_-), \, f = \varepsilon_\star, \, x = a_+ + a_- \; \mbox{\rm and} \; y = b_+ + b_-.
$$
Note that  $\jmath^2 = -e$, hence $\langle e, \jmath\rangle_\RR \cong \CC$. Moreover,
$$
x \jmath  y = (a_+ + a_-) \jmath (b_+ + b_-) = 0,
$$
since $a_+ b_+  = a_- b_-$.  Put:
$$
x_1 = x, \; y_1 = y, \; x_2 = x \jmath \; \mbox{\rm and}\; y_2 = -\jmath y.
$$
Then we have: $x_1 y_1 = x_2 y_2$, $x_1 y_2 = 0 = x_2 y_1$. Identifying the elements $e, f, \jmath, x_1, x_2, y_1, y_2$ with the corresponding elements given by (\ref{E:KeyElements1}) and 
(\ref{E:KeyElements}), we obtain the anticipated isomorphism $O^G \cong A$. 
\end{remark}

\begin{theorem}\label{T:BlockDescr} The $\RR$-linear abelian categories $\HC_\circ(\lieg, K)$ and  $\Rep(A)$ are equivalent.
\end{theorem}

\begin{proof} First note that the functor 
$$
\HC(\lieg, K)_{\CC} \lar \HC(\tilde{\lieg}, \tilde{K}), \; M \mapsto \CC \otimes_{\RR} M
$$
is fully faithful; see Lemma \ref{L:ExtScalars}. Moreover, it induces a $\CC$-linear  equivalence of categories $\bigl(\HC(\lieg, K)_{\CC}\bigr)^\omega \simeq 
\HC(\tilde{\lieg}, \tilde{K})$ (see Lemma \ref{L:ExtScalars2}), which restricts to a $\CC$-linear  equivalence $\bigl(\HC_\circ(\lieg, K)_{\CC}\bigr)^\omega \simeq 
\HC_\circ(\tilde{\lieg}, \tilde{K})$.  Let $G = \mathsf{Gal}(\CC/\RR) = \langle \sigma | \sigma^2 = e\rangle$. Then $G$ acts on $\HC_\circ(\lieg, K)_{\CC}$, hence also on 
$\bigl(\HC_\circ(\lieg, K)_{\CC}\bigr)^\omega\simeq \HC_\circ(\tilde{\lieg}, \tilde{K})$  and the corresponding auto-equivalence $\Phi_\sigma: \HC_\circ(\tilde{\lieg}, \tilde{K}) \lar \HC_\circ(\tilde{\lieg}, \tilde{K})$ is just the functor $(-)^\dagger$ from Lemma \ref{L:ComplexConj}.

By Corollary \ref{C:GaloisExt} and Remark \ref{R:SubtletiyGroupAct}, we have the following $\RR$-linear equivalences of categories:
$$
\HC_\circ(\lieg, K) \simeq \bigl(\HC_\circ(\lieg, K)_{\CC}[G, \Phi]\bigr)^\omega \simeq 
\bigl(\bigl(\HC_\circ(\lieg, K)_{\CC}\bigr)^\omega[G, \Phi]\bigr)^\omega \simeq \bigl(\HC_\circ(\tilde\lieg, \tilde{K})[G, \Phi]\bigr)^\omega. 
$$
Next, by Lemma \ref{L:ComplexConj}, Theorem \ref{T:DFO} and Lemma \ref{L:MoritaEq} we have $\RR$-linear equivalences of categories: 
$$
\bigl(\HC_\circ(\tilde\lieg, \tilde{K})[G, \Phi]\bigr)^\omega \simeq \bigl(\Rep(O)[G, \Psi]\bigr)^\omega \simeq O[G, \phi]\mbox{--}\mathsf{fdmod} \simeq \Rep(A).
$$
Here,  the action of $G$ on the category $\Rep(O)$ is given by the auto-equivalence $\Psi_\sigma = \sigma^\sharp: \Rep(O) \lar \Rep(O)$.
This proves the statement.
\end{proof}

\begin{remark}
A recent work of Januszewski \cite{Janusz} provides another perspective on the description of the principal block $\HC_\circ(\lieg, K)$. 
\end{remark}

\end{document}